 \def\misajour{$3/5/2014$}
\documentclass[10pt]{article}

\usepackage[applemac]{inputenc}
\usepackage[pdftex]{graphicx}

\usepackage{amsmath}

\usepackage{amsthm}

\usepackage[applemac]{inputenc}

\usepackage[pdftex]{graphicx}
\usepackage {hyperref}

\def\C{\mathbf {C}}
\def\Q{\mathbf{ Q}}
\def\R{\mathbf {R}}
\def\Z{\mathbf {Z}}
\def\ZK{\Z_K}

\def\Atilde{\tilde{A}}
\def\Btilde{\tilde{B}}

\def\rmN{\mathrm{N}}
\def\calE{{\mathcal{E}}}
\def\rmM{\mathrm{M}}
\def\rmh{\mathrm{h}}
\def\tors{{\mathrm{tors}}}
\def\Imag{{\mathrm{Im}}}

\def\house#1{\setbox1=\hbox{$\,#1\,$}%
\dimen1=\ht1 \advance\dimen1 by 2pt \dimen2=\dp1 \advance\dimen2 by 2pt
\setbox1=\hbox{\vrule height\dimen1 depth\dimen2\box1\vrule}%
\setbox1=\vbox{\hrule\box1}%
\advance\dimen1 by .4pt \ht1=\dimen1
\advance\dimen2 by .4pt \dp1=\dimen2 \box1\relax}

\newtheorem{theoreme}[equation]{\indent Theorem}
\newtheorem{lemme}[equation]{\indent Lemma}
\newtheorem{corollaire}[equation]{\indent Corollary}
\newtheorem{proposition}[equation]{\indent Proposition}

\newcounter{compteurkappa}

\def\Newcst#1{
\refstepcounter{compteurkappa}
\kappa_{
\arabic{compteurkappa}}
\label{#1}
}

\def\cst#1{\kappa_{\ref{#1}}}
\def\boxit#1#2{\setbox1=\hbox{\kern#1{#2}\kern#1}%
\dimen1=\ht1 \advance\dimen1 by #1 \dimen2=\dp1 \advance\dimen2 by #1
\setbox1=\hbox{\vrule height\dimen1 depth\dimen2\box1\vrule}%
\setbox1=\vbox{\hrule\box1\hrule}%
\advance\dimen1 by .4pt \ht1=\dimen1
\advance\dimen2 by .4pt \dp1=\dimen2 \box1\relax}

\begin{document}

 \hfill

 \null
 \vskip -3 true cm

 \hfill
 {\it \misajour}

 \bigskip

\begin{center}
 \smallskip
{ \Large {\bf
Solving simultaneously
Thue equations}}

\smallskip
{\Large {\bf
in the almost totally imaginary case}}
\bigskip

{\large {\sc Claude Levesque} and {\sc
Michel Waldschmidt}}

\end{center}

%%%%%%%%%%%%%%%%%%%%%%%%%%%%%%%%%%%%%%%%%%%%%%%%
\section*{Abstract}
Let $\alpha$ be an algebraic number of degree $d\ge 3$ having at most one real conjugate and let $K$ be the
algebraic number field $\Q(\alpha)$. For any unit $\varepsilon$ of $K$ such that $\Q(\alpha\varepsilon)=K$,
we consider the irreducible polynomial $f_\varepsilon(X)\in\Z[X]$ such that $f_\varepsilon(\alpha\varepsilon)=0$.
Let $F_\varepsilon(X,Y)\ = Y^df_\varepsilon(X/Y)\in\Z[X,Y]$ be the associated binary form. For each positive integer $m$,
we exhibit an effectively computable bound for the solutions $(x,y,\varepsilon)$ of the diophantine equation $|F_\varepsilon(x,y)|\leq m$.

%%%%%%%%%%%%%%%%%%%%%%%%%%%%%%%%%%%%%%%%%%%%%%%%
\section{The main theorem}\label{S:ThPpal}
\mbox{} \indent
Let $\alpha$ be an algebraic number of degree $d\ge 3$ over $\Q$. Denote by $K$ the algebraic number field $\Q(\alpha)$,
by $f\in \Z[X]$ the irreducible polynomial of $\alpha$ over $\Z$, by $ \ZK^\times$ the unit group of $K$. For any unit $\varepsilon\in \ZK^\times$ such that $\Q(\alpha\varepsilon)=K$, we denote by $f_{\varepsilon}(X)\in \Z[X]$ the irreducible polynomial of
$\alpha\varepsilon$ over $\Z$ (unique if we require the leading coefficient to be $>0$) and by $F_{\varepsilon}(X,Y)\in \Z[X,Y]$
the irreducible binary form associated to $f_{\varepsilon}(X)$ via the condition $F_{\varepsilon}(X,1)=f_{\varepsilon}(X)$. \medskip

A particular case of Corollary 3.6 of \cite{LW1} (dealing with Thue--Mahler equations, though Thue equations are involved in this paper),
is the following one.

\begin{theoreme}

Let $m$ be a rational positive integer. Then the set $\calE$, where
$$
\calE=
\bigl\{
(x,y,\varepsilon)\in\Z^2\times\ZK^\times\; \mid
\;
xy\not=0,
\;
\Q(\alpha\varepsilon)=K
 \;\;\hbox{and}\;\; |F_\varepsilon(x,y)|\le m\bigr\},
$$
is finite.
\end{theoreme}

The proof in \cite{LW1}, involving Schmidt's subspace theorem, allows to give an effectively computable upper bound for the number of elements of
 $\calE$ as a function of $m$, $d$ and the absolute logarithmic height $\rmh(\alpha)$ of $\alpha$ (whose definition will be reminded below),
 but does not allow to give a bound for the heights of the elements of $\calE$. The main result of this paper gives an effectively computable
upper bound for the solutions in the particular case where the algebraic number field $K$ has at most one real embedding into $\C$.

\begin{theoreme}\label{theoreme:principal}
Suppose that the algebraic number field $K$ has at most one real embedding. Then there exists an effectively computable constant $\Newcst{kappa1}>0$,
 depending only on $\alpha$, such that, for all $m\ge 2$, each solution $(x,y,\varepsilon)\in\Z^2\times\ZK^\times$ of the
 Thue inequality $| F_\varepsilon(x,y)|\leq m$ with $xy\neq 0$ and $\Q(\alpha\varepsilon)=K$ satisfies
$$
\max\{|x|,\; |y|,\; e^{\rmh(\alpha\varepsilon)}\}\le m^{\cst{kappa1}}.
$$
\end{theoreme}

The conclusion of  Theorem $\ref{theoreme:principal}$ can be stated with the use of the norm $\rmN_{K/\Q}$, since $\rmN_{K/\Q}(X-\alpha\varepsilon Y)=a_0^{-1}F_\varepsilon(X,Y)$: {\it for
$(x,y,\varepsilon)\in\Z^2\times\ZK^\times$ with $xy\neq 0$ and $ \Q(\alpha\varepsilon)=K$, one has
$$
|\rmN_{K/\Q}(x-\alpha\varepsilon y)|
\ge
\Newcst{kappabis}
\max\{|x|,\; |y|,\; e^{\rmh(\alpha\varepsilon)}\}^{\Newcst{kappa1ter}}
$$
with two effectively computable positive constants $\cst{kappabis}$ and $\cst{kappa1ter}$ depending only on $\alpha$. }
\medskip

If the algebraic number field $K$ has a unique real embedding into $\C$, the degree $d$ of $K$ is odd and $K$ has $(d-1)/2$ pairs of
 complex embeddings.
In the particular case $d=3$, namely for a cubic field $K$, the hypothesis that there is only one real embedding boils down to requiring
that the unit rank of $K$ be $1$, and this particular case of Theorem $\ref{theoreme:principal}$ was obtained in \cite{LW2}.
\medskip

 When $K$ has no real embedding, namely when $K$ is totally imaginary, the degree $d$ of $K$ is even and $K$ has $d/2$ pairs of
 complex embeddings. It looks like the totally imaginary case of Theorem $\ref{theoreme:principal}$ is not trivial and requires a
 diophantine argument. Only the case of a Thue equation is elementary (and not a family of such equations).

\begin{lemme}\label{Lemme:ThueTotalementImaginaire}
Let $F(X,Y)\in\Z[X,Y]$ be a binary form of degree $d$ with integer coefficients and without a real root, (so $F$ is a product over $\R$ of
 definite positive quadratic forms). Then there exists an effectively computable constant $\Newcst{kappa2}$ such that, for each $m>0$, each solution $(x,y)\in\Z^2$ of the Thue inequality
$$
|F(x,y)|\leq m
$$
satisfies
$$
\max\{|x|,\; |y|\}\leq \cst{kappa2} m^{1/d}.
$$
\end{lemme}
\indent \begin{proof}[\indent Proof]
Write
$$
F(X,Y)=a_0\prod_{j=1}^{n} (X-\alpha_jY)(X-\overline{\alpha_j}Y),
$$
where $n=d/2$, where $a_0$ is the leading coefficient of $F(X,1)$, while $\alpha_j$ and $\overline{\alpha_j}$ $(j=1,\dots,n)$ are the roots of $F(X,1)$. Since
$$
|x-\alpha_jy|=|x-\overline{\alpha_j}y|\ge |\Imag(\alpha_j)|\cdot |y|,
$$
where $\Imag(z)$ is the imaginary part $(z-\overline{z})/2i$ of the complex number $z$,
we have
$$
|y| \le \Newcst{kappa2bis} m^{1/d}
\quad\hbox{with}\quad
\cst{kappa2bis}= |a_0 |^{-1/d} \prod_{j=1}^{n} |\Imag(\alpha_j)|^{-2/d}.
$$
If
$$
|x|\le 2\left(\max_{1\le j\le n} |\alpha_j|\right) |y|,
$$
the proof of Lemma $\ref{Lemme:ThueTotalementImaginaire}$ is complete. Otherwise, we have $|x-\alpha_j y|\ge |x|/2$
for $j=1,\dots,n$, and we get
$$
|x|\le 2\left(
\frac{m}{|a_0|}\right)^{1/d},
$$
which concludes the proof.
\end{proof}

The hypothesis that there is at most one real embedding will play an essential role in the proof of Theorem $\ref{theoreme:principal}$. However, a significant part of the proof is valid without this hypothesis, and we plan to pursue our work on the subject. \medskip

% -+-~-+-~-+-~-+-~-+-~-+-~-+-~-+-~-+-~-+-~-+-~-+-~-+-~-+-~-+-~-+-~-+-~-+-~-+-~-+-~-+-~-+-~-+-~-+-~-+-~-+-~-+-~-+-~-+-~-+-~
%\vfill\newpage

\noindent
{\bf Example.} Let $D\geq 2$, $c\in \{1, -1\}$ and  let $K={\bf
Q}(\theta)$ with $ \theta = (1+i)\root 4 \of { D^4+c}$. It is a result of Stender
\cite{St} that
 $$\alpha = \varepsilon = D^2+D\theta+\frac12 \theta^2$$
  is the fundamental unit  of $K$. Notice that $ \theta ^4= -4(D^4+c)$ and $|\varepsilon|>1$.
The minimal polynomial of $\varepsilon$ is
$$
f(X) = X^4-4D^2X^3+(8D^4+2c)X^2+4cD^2X +1
$$
with
$$
f(X)=(X-\varepsilon_1)(X-\varepsilon_2)(X-\varepsilon_3)
(X-\varepsilon_4)
$$
where   $\varepsilon_1,\; \varepsilon_2,\; \varepsilon_3,\;
\varepsilon_4\;$  are  the conjugates  of $\alpha=\varepsilon_1$. Since $X^4f(-c/X)=f(X)$, we can choose $ \varepsilon_3=-c \varepsilon_1^{-1}$ and
$ \varepsilon_4=-c \varepsilon_2^{-1}$.
For $n\in {\bf Z}$, define $f_n(X)$, $a_n$, $b_n$,
$c_n$ by
$$
X^4+a_nX^3+b_nX^2+c_nX +1
= (X- \varepsilon_1^n)(X- \varepsilon_2^n) (X-
\varepsilon_3^n)(X- \varepsilon_4^n)
$$
and
$$
f_n(X)= (X- \varepsilon_1^{n+1})(X- \varepsilon_2^{n+1}) (X-
\varepsilon_3^{n+1})(X- \varepsilon_4^{n+1}),
$$
so that
$$
f_n(X)=X^4+a_{n+1}X^3+b_{n+1}X^2+c_{n+1}X +1.
$$
Using once more the relation
$X^4f(-c/X)=f(X)$, we deduce
$$
c_n=a_{-n}=(-c)^n a_n\quad \hbox{and}\quad b_{-n}=b_n
$$
for $n\in\Z$.
The sequence $(a_n)_{n\in\Z}$ satisfies the linear recurrence equation
$$
a_{n+4}-4D^2a_{n+3}+(8D^4+2c)a_{n+2}+4cD^2a_{n+1}+a_n=0
$$
with the initial conditions $a_{-1}=4cD^2$, $a_0=-4$, $a_1=-4D^2$, $a_2=4c$. The sequence $(c_n)_{n\in\Z}$ is given by $c_n=a_{-n}$ for $n\in\Z$, while the  sequence $(b_n)_{n\in\Z}$ satisfies the linear recurrence equation
\begin{align}\notag
b_{n+6}-(8D^4+2c)b_{n+5} &- (16cD^4+1) b_{n+4} \\[1mm] \notag
 - (16D^4-4c) b_{n+3}& - (16cD^4+1) b_{n+2} -(8D^4+2c)b_{n+1} +b_n=0,
\end{align}
with the initial conditions
$$\left\{
\begin{array} {cll} 
b_{-2}&=&64 D^8 +64 cD^4 +6,\\[1.5mm]
b_{-1}&=& 8 D^4 +2c,\\[1.5mm]
b_{0}&=&6,\\[1.5mm]
b_{1}&=& 8 D^4 +2c,\\[1.5mm]
b_2&=&64 D^8 +64 cD^4 +6,\\[1.5mm]
b_3&=&512 D^{12} + 1768 D^8c+ 264 D^4 +2c.
\end{array}\right.
$$
Set
$$F_n(X,Y)= Y^4f_n(X/Y),
$$
so that
$$
F_n(X,Y)
=(X-\varepsilon_1^{n+1}Y)(X-\varepsilon_2^{n+1}Y)
(X-\varepsilon_3^{n+1}Y) (X-\varepsilon_4^{n+1}Y).
$$
Let $m$ be a rational positive  integer. Denote by
$${\calE} \;=\; \{ (x,y,n)\in {\bf Z}^3 \,|\, xy\neq 0, \;n\neq -1
 \mbox{ and } |F_n(x,y)| \leq m\}
 $$
the set of solutions of  the Thue inequality
$|F_n(X,Y)| \;\leq\; m$.
 Then from Theorem $\ref{theoreme:principal}$ we deduce the upper bound
 $$
\max\{|x|,\; |y|,\; |\varepsilon|^{ |n|}  \mid (x,y,n)\in {\calE} \} \le \Newcst{kappaexample1} m^{\Newcst{kappaexample2}},
$$
where $\cst{kappaexample1}$ and $\cst{kappaexample2}$ are positive  constants depending only on $D$.

% -+-~-+-~-+-~-+-~-+-~-+-~-+-~-+-~-+-~-+-~-+-~-+-~-+-~-+-~-+-~-+-~-+-~-+-~-+-~-+-~-+-~-+-~-+-~-+-~-+-~-+-~-+-~-+-~-+-~-+-~
%  \vfill\newpage
 \bigskip\bigskip

Let us denote by $\Phi=\{\sigma_1,\ldots,\sigma_d\}$ the set of embeddings of $K$ into $\C$. Let us write the irreducible polynomial $f$ of $\alpha$ over $\Z$ as
$$
f(X)=a_0X^d+a_1X^{d-1}+\cdots+a_d\in\Z[X],
$$
so
$$
f(X)= a_0\prod_{i=1}^d \bigl(X-\sigma_i(\alpha)\bigr),
$$
 and the irreducible binary form associated to
$F\in\Z[X,Y]$ is
$$
F(X,Y)=Y^d f(X/Y)=a_0X^d+a_1X^{d-1}Y+\cdots+a_dY^d.
$$
For $\varepsilon\in\ZK^\times$ satisfying $\Q(\alpha\varepsilon)=K$, we have
 $$
F_\varepsilon (X,Y)=a_0\prod_{i=1}^d \bigl(X-\sigma_i(\alpha\varepsilon)Y\bigr)\in\Z[X,Y].
$$

Let us denote by $\rmh$ the absolute logarithmic height and by $\rmM$ the Mahler measure. Hence,
$$
\rmh(\alpha)=\frac{1}{d}\log \rmM(\alpha)
\quad
\hbox{where}
\quad
\rmM(\alpha)=
a_0
 \prod_{1\le i\le d} \max\{1,|\sigma_i(\alpha)|\}.
$$

%%%%%%%%%%%%%%%%%%%%%%%%%%%%%%%%%%%%%%%%%%%%%%%%
\section{Tools}\label{S:Outils}
\mbox{} \indent
 In this section, let us put together the auxiliary lemmas which will prove useful. We will use some results in geometry of numbers
to establish an equivalence of norms (Lemma $\ref{Lemme:PlongementAdapte}$). Then we state in Lemma 6 what is Lemma 2 of
 \cite{LW2}. Finally in
Proposition $\ref{Proposition:FormeLineaireLogarithmes}$ and in Corollaries $\ref{Cor:MinorationCLL1}
$ and
$\ref{Cor:MinorationCLL2}$ we exhibit some lower bounds of linear forms of logarithms of algebraic numbers.

\subsection{Equivalence of norms}\label{S:EquivalenceNormes}

\mbox{} \indent
Let $K$ be an algebraic number field of degree $d$. We denote by $\epsilon_1,\ldots,\epsilon_r$ a basis of the unit group
of $K$ modulo $K^\times_{\tors}$ and we suppose $r\geq 1$. Let us recall that the house of the algebraic number $\gamma$,
denoted $\house{\gamma}\; $, is the maximum of the absolute values of the conjugates of $\gamma$.
\medskip

Let us first remark that there exists a positive constant $\Newcst{hauteurgamma}$, depending only on
 $\epsilon_1,\ldots,\epsilon_r$, such that, if $c_1,\ldots, c_r$ are rational integers and if we let
$$
\gamma= \epsilon_1^{c_1}\cdots \epsilon_r^{c_r}
\quad
\hbox{with} \quad C=\max\{
|c_1|,\ldots,|c_r|\}
$$
be a unit of $K$, then
\begin{equation}\label{Equation:EstimationsTriviales}
e^{-\cst{hauteurgamma} C }\le |\varphi(\gamma)|\le e^{ \cst{hauteurgamma} C }
\end{equation}
 for any embedding $\varphi$ of $K$ into $\C$. More precisely, the inequalities 
 $$
 |\varphi(\gamma)|\le \prod_{i=1}^r {\house{\epsilon_i}\,}^C
\quad\hbox{and}\quad
\left |\varphi\left(\gamma^{-1}\right) \right|\le \prod_{i=1}^r {\house{\epsilon_i}\,}^C
$$
(note that $\house{\epsilon_i}\ge 1$)
suggest to take
$$
\cst{hauteurgamma} =
\sum _{i=1}^r \log\house{\epsilon_i}.
$$

The following result, Lemma $\ref{Lemme:PlongementAdapte}$, is a variant of Lemma 5.1 of \cite{ST}. It shows that the two inequalities of ($\ref{Equation:EstimationsTriviales}$) are optimal. They make $C$ appear as an upper bound, while the two inequalities of the conclusion of Lemma $\ref{Lemme:PlongementAdapte}$
 make $C$ appear as a lower bound.

\begin{lemme}\label{Lemme:PlongementAdapte}
There exists an effectively computable positive constant $\Newcst{kappa:PlongementAdapte}$, depending only on $\epsilon_1,\ldots,\epsilon_r$,
 with the following property. If $c_1,\ldots, c_r$ are rational integers and if we let
$$
\gamma= \epsilon_1^{c_1}\cdots \epsilon_r^{c_r},\quad C=\max\{
 |c_1|,\ldots,|c_r|\},
$$
then there exist two embeddings $\sigma$ and $\tau$ of $K$ into $\C$ such that 
$$
|\sigma(\gamma) |\ge e^{\cst{kappa:PlongementAdapte}C }
\quad\hbox{and}\quad
|\tau(\gamma) |\le e^{-\cst{kappa:PlongementAdapte} C}.
$$
\end{lemme}

Obviously, $\cst{kappa:PlongementAdapte}\leq \cst{hauteurgamma}$.

\begin{proof}[\indent Proof]
For any embedding $\varphi$ of $K$ into $\C$, we have
$$
\log |\varphi(\gamma)| = c_1 \log |\varphi(\epsilon_1)|+\cdots + c_r\log |\varphi(\epsilon_r)|.
$$
The matrix $\bigl(\log |\varphi(\epsilon_i)|\bigr)$,
with $d$ lines and $r$ columns,
where the indices for the lines are the $\varphi$'s of $\Phi$ and the indices for  the columns
are the $j$'s of the set $\{ 1, 2, \dots , r \}$, has rank $r$. Therefore we can write
$c_1,\ldots,c_r$ as linear combinations of the elements $\log |\varphi(\gamma) |$, whose coefficients in absolute values have upper bounds which
are functions of the regulator of $K$. Hence
$$
 \max\{ |c_1|,\ldots,|c_r|\} \le \Newcst{kappa:PlongementAdapte3} \max_{\varphi\in \Phi} \bigl(\log |\varphi(\gamma)| \bigr).
$$

Let us take for $\sigma$ an element of $\Phi $ such that $|\sigma(\gamma)|$ be maximal:
$$
|\sigma(\gamma)|= \max_{\varphi\in \Phi} |\varphi(\gamma)|.
$$
This leads to
$$
C\le \cst{kappa:PlongementAdapte3} \log |\sigma(\gamma)|,
$$
providing the conclusion for $\sigma$ with $ \cst{kappa:PlongementAdapte}= \cst{kappa:PlongementAdapte3}^{-1}$.
\medskip

We next apply this result to
$$
\gamma^{-1}=\epsilon_1^{-c_1}\cdots \epsilon_r^{-c_r}.
$$
If $\tau$ is an element of $\Phi $ such that $|\tau(\gamma)|$ be minimal, namely
$$
|\tau(\gamma)|= \min_{\varphi\in \Phi} |\varphi(\gamma)|,
$$
then we have
$$
C\le \cst{kappa:PlongementAdapte3} \log |\tau\left(\gamma^{-1}\right)|.
$$  \end{proof}

\indent
{\bf Remark}.
Under the hypotheses of Lemma $\ref{Lemme:PlongementAdapte}$, if $\gamma_0$ is a nonzero element of $K$ and if we let
$\gamma_1=\gamma_0\gamma$, we deduce
$$
e^{-\cst{hauteurgamma} C }\le
\min_{\varphi\in\Phi} \left| \varphi\left(|\gamma_0| ^{-1} \right) \varphi(\gamma_1) \right|
\le
 e^{-\cst{kappa:PlongementAdapte} C }
$$
and
$$
e^{\cst{kappa:PlongementAdapte}C }
\le
\max_{\varphi\in\Phi}\left| \varphi\left(|\gamma_0| ^{-1} \right) \varphi(\gamma_1) \right|
\le
e^{ \cst{hauteurgamma} C }.
$$

To benefit from these inequalities, we will use the estimates
$$
 |\varphi(\gamma_0)|\le \house{\gamma_0}\le e^{d\rmh(\gamma_0)}
\quad
\hbox{and}
\quad
\rmh\left(\gamma_0^{-1}\right)=\rmh(\gamma_0)
$$
from which we deduce
$$
e^{-\cst{hauteurgamma} C -d\rmh(\gamma_0)}\le
\min_{\varphi\in\Phi} | \varphi(\gamma_1)|
\le
 e^{-\cst{kappa:PlongementAdapte} C +d \rmh(\gamma_0)}
$$
and
$$
e^{\cst{kappa:PlongementAdapte}C -d\rmh(\gamma_0) }
\le
\max_{\varphi\in\Phi} |\varphi(\gamma_1)|
\le
e^{ \cst{hauteurgamma} C +d\rmh(\gamma_0) }.
$$
The term $d\rmh(\gamma_0)$ always appears as an upper bound, and as an error term.
Note also that the constant $\cst{hauteurgamma} $ is large (it appears in the upper bounds), while the constant $\cst{kappa:PlongementAdapte}$ is small (it comes into play in some lower bounds).

\subsection{Upper bound
involving the norm
} \label{SS:norme}

Given an algebraic number $\gamma$ of degree $\leq d$ and norm $\leq m$, there exists a unit $\varepsilon$ in the field $\Q(\alpha)$ such that the conjugates of $\varepsilon\gamma$ are bounded from above by a constant times $m^{1/d}$.
This is a consequence of Lemma A.15 of \cite{ST}, a result which we want to state; (this is also Lemma 2 of \cite{LW2}).

 \begin{lemme}\label{Lemme: LemmaA.15ST}
Let $K$ be an algebraic number field of degree $d$ with regulator $R$ and let $\gamma$ be a nonzero element of $\ZK$ whose
norm in absolute value is $m$. Then there exists a unit $\varepsilon\in\ZK^\times$ such that
 $$
 \frac{1}{R}
 \max_{1\le j\le d}
 \left|
 \log (m^{-1/d}
 |\sigma_j(\varepsilon\gamma)|)
 \right|
 $$
 is bounded  from above by an effectively computable constant depending only on $d$.
 \end{lemme}

{\bf Remark.} The unit $\varepsilon\in\ZK^\times$ may come from the group generated by the basis of the unit group modulo the torsion.
In other words, the torsion elements do not come into play because of the absolute values which appear in $|\sigma_j(\varepsilon\gamma)|$.
\medskip

We use this lemma in the same way we did in \cite{LW2}: there exists an effectively computable constant $\Newcst{LemmaA.15ST}$, which is a function of
 $d$ and $R$, such that, if $\gamma$ is a nonzero element of $\ZK$ whose norm has an absolute value $\leq m$, then there exists a unit $\varepsilon\in\ZK^\times$ such that
$$
 \max_{1\le j\le d}
 \left|
 \sigma_j(\varepsilon\gamma)
 \right|
 \le
 \cst{LemmaA.15ST}
m^{1/d}.
$$
We will suppose $m\geq 2$ and we will rather use the weaker upper bound
\begin{equation}\label{Equation:lemmeA.15ST}
 \max_{1\le j\le d}
 \left|
 \sigma_j(\varepsilon\gamma)
 \right|
 \le
m^{ \cst{LemmaA.15STbis}}
 \end{equation}
 with $\Newcst{LemmaA.15STbis}$, an effectively computable constant which can be calculated as a function of $d$ and $R$.

\subsection{Diophantine tools}\label{SS:outilsdiophantiens}

 \mbox{} \indent
In this section, we are given two positive integers $s$ and $D$; the constants $\Newcst{kappa:FLL}$, $\Newcst{kappa:FLL1}$ and $\Newcst{kappa:FLL2}$ depend only on $s$ and $D$ and are effectively computable.
\medskip

Here are the hypotheses and notations common to Proposition $\ref{Proposition:FormeLineaireLogarithmes}$ and to Corollaries $\ref{Cor:MinorationCLL1}$ and $\ref{Cor:MinorationCLL2}$.
Let $\gamma_1,\ldots, \gamma_s$ be algebraic numbers in an algebraic number field of degree $\leq D$. Let $c_1,\ldots,c_s$ be
rational integers. Suppose $\gamma_1^{c_1}\cdots \gamma_s^{c_s} \not=1$. Moreover, let $H_1,\ldots,H_s$ be real numbers $\geq 1$ which satisfy
$$
H_j\ge \rmh(\gamma_j) \quad (1\le j\le s).
$$
We will use the following particular case of Theorem 9.1 of \cite{GL326} (see Proposition 9.21 of \cite{GL326}).

\begin{proposition}\label{Proposition:FormeLineaireLogarithmes}
Let $\lambda_1,\dots,\lambda_s$ be complex numbers such that   
$\gamma_j=e^{\lambda_j}$  for all  $ j\in\{ 1, \dots ,s\}$. Put
$$
\Lambda=c_1\lambda_1+\cdots+c_s\lambda_s.
$$
Suppose $H_j\ge e^{|\lambda_j|}$ for $j=1,\dots,s$. Let $C_0$ be a real number satisfying
$$
 C_0\ge 2, \quad
C_0\ge
\max_{1\le j<s}
\left\{
\frac{|c_s|}{ H_j}+
\frac{|c_j|}{ H_s}\right\}
\cdotp
$$
Then
$$
|\Lambda|\ge\exp\bigl\{- \cst{kappa:FLL} H_1 \cdots
H_s \log C_0 \bigr\}.
$$
\end{proposition}

Note that the hypothesis $\gamma_1^{c_1}\cdots \gamma_s^{c_s} \not=1$ implies $\Lambda \not=0$.
 Proposition $\ref{Proposition:FormeLineaireLogarithmes}$ will be used via the following corollary.

 \begin{corollaire}\label{Cor:MinorationCLL1}
Suppose $H_j\le H_s$ for $1\le j\le s$.
Let $C_1$ a real number satisfying
$$
C_1\ge 2, \quad
C_1\ge
\max_{1\le j\le s} \left\{
 \frac{H_j}{H_s} |c_j|\right\}.
$$
Then
$$
|\gamma_1^{c_1}\cdots \gamma_s^{c_s}-1|>
\exp\{-
\cst{kappa:FLL1} H_1\cdots H_s\log C_1\}.
$$
 \end{corollaire}

\indent \begin{proof}[\indent Proof]
Under the hypotheses of Corollary $\ref{Cor:MinorationCLL1}$, we choose an embedding of the algebraic number field
$K=\Q(\gamma_1,\dots,\gamma_s)$ into $\C$. From the definition of the absolute logarithmic height $\rmh$,
we deduce $|\gamma_j|\le e^{D \rmh(\gamma_j)}$ for $j=1,\dots,s$. We introduce some real numbers
$\nu_1,\dots,\nu_s$ via the conditions
$$
\gamma_j=|\gamma_j| e^{2i\pi \nu_j}, \quad -1<\nu_j\le 1 \qquad (1\le j\le s)
$$
and we put
$$
\lambda_j=\log |\gamma_j|+2i\pi \nu_j \qquad (1\le j\le s).
$$
In other words, $\lambda_j$ is the main determination of the logarithm of $\gamma_j$. In particular,
 $e^{\lambda_j}=\gamma_j$ and $|\lambda_j|\le DH_j+2\pi$ ($1\le j\leq s$).
Write
$$
\Lambda_0=c_1\lambda_1+\cdots+c_s\lambda_s.
$$
So
$$
e^{\Lambda_0}=\gamma_1^{c_1}\cdots \gamma_s^{c_s}\not=1.
$$
Note also $\Lambda$ the main determination of the logarithm of $\gamma_1^{c_1}\cdots \gamma_s^{c_s}$. So there exists
$c_0\in\Z$ such that $\Lambda=2i\pi c_0+\Lambda_0$.
We may suppose $|e^{\Lambda}-1|<1/2$; otherwise, the conclusion of Corollary $\ref{Cor:MinorationCLL1}$ is trivial; we deduce (for
instance, see Lemma 3 of \cite{LW2})
$$
|\Lambda|\le 2 |\gamma_1^{c_1}\cdots \gamma_s^{c_s}-1|.
$$
In the same vein, write
$\lambda_0=2i\pi$, $\gamma_0=1$, $H_0=2\pi$, so
$$
\Lambda=c_0\lambda_0+c_1\lambda_1+\cdots+c_s\lambda_s.
$$
From the upper bound $|\Lambda|\le 1$, we deduce
$$
2\pi |c_0|\le 1+\sum_{j=1}^s |c_j \lambda_j|\le 1+\sum_{j=1}^s ( DH_j+2\pi) |c_j |.
$$
Therefore the inequalities
$$
C_1\ge
 \frac{H_j}{H_s} |c_j|,
 \quad H_j\ge 1 \qquad (1\le j\le s)
$$
and
$$
\frac{|c_0|}{H_s} \le \left( \frac{1}{2\pi}+\frac{sD}{2\pi} +s\right) C_1 \le 2s D C_1
$$
allow us to use Proposition $\ref{Proposition:FormeLineaireLogarithmes}$ with $s$ replaced by $s+1$ and
 $4sDC_1$ playing the role of $C_0$. This concludes the proof.
\end{proof}

The following particular case of Corollary $\ref{Cor:MinorationCLL1}$ is also deduced from Corollary 9.22 of \cite{GL326}.

 \begin{corollaire}\label{Cor:MinorationCLL2}
Let $C_2$ be a real number satisfying
$$
C_2\geq\max \{ 2, |c_1|,\ldots, |c_s| \}.
$$
Then
$$
|\gamma_1^{c_1}\cdots \gamma_s^{c_s}-1|>
\exp\{-
\cst{kappa:FLL2} H_1\cdots H_s\log C_2\}.
$$
 \end{corollaire}

%%%%%%%%%%%%%%%%%%%%%%%%%%%%%%%%%%%%%%%%%%%%%%%%
\section{The reciprocal polynomial}\label{S:PolynomeReciproque}

 \mbox{} \indent
\mbox{} \indent Let us show that there is no restriction, for the statement of Theorem $\ref{theoreme:principal}$, to suppose $|x|\leq |y|$.
 Suppose that Theorem $\ref{theoreme:principal}$ is proved in the particular case $|x|\leq |y|$ with the constant $\cst{kappa1}$ replaced by
 $\Newcst{kappa1bis}(\alpha)$. Suppose now that we are under the hypotheses of Theorem $\ref{theoreme:principal}$ and consider a solution $(x,y,\varepsilon)\in\Z^2\times\ZK^\times$ of the inequality $| F_\varepsilon(x,y)|\le m$ with $xy\neq 0$ and $|x|> |y|$. Write $\alpha'= \alpha^{-1}$; let $g\in\Z[Y]$ be the irreducible polynomial of $\alpha'$:
$$
g(Y)=Y^df\left(Y^{-1}\right)=a_d\prod_{i=1}^d \bigl(Y-\sigma_i(\alpha')\bigr). 
$$
Put
$$
 \varepsilon'= \varepsilon^{-1}, \quad g_{\varepsilon'}(Y)=a_d\prod_{i=1}^d \bigl(Y-\sigma_i(\alpha'\varepsilon')\bigr),
\quad
G_{\varepsilon'}(Y,X)=X^dg_{\varepsilon'}(Y/X),
$$
so that
$$
G_{\varepsilon'}(Y,X)=F_\varepsilon (X,Y).
$$
Consequently, $(y,x,\varepsilon')$ is a solution of the Thue inequality
$$
|G_{\varepsilon'} (y,x)|\le m.
$$
Since $\rmh\bigl((\alpha\varepsilon)^{-1}\bigr)=\rmh(\alpha\varepsilon)$, we deduce from Theorem $\ref{theoreme:principal}$
applied to $G_{\varepsilon}$, with the condition $|y|<|x|$,
$$
\max\{|x|,\; e^{\rmh(\alpha\varepsilon)}\}\le m^{\cst{kappa1bis}(\alpha^{-1})},
$$
which allows us to draw the conclusion, by taking $\cst{kappa1}=\max\{\cst{kappa1bis}(\alpha), \cst{kappa1bis}(\alpha^{-1})\}$.
\medskip

From now on, we fix an element $(x,y,\varepsilon)\in \calE$ with $|x|\leq |y|$. The constants $ \cst{majhauteuralphaepsilon}, \cst{minhauteuralphaepsilon},\dots,\cst{kappafinal}$ which will appear, depend only on $\rmh(\alpha)$, $d$ and $m$;
they are easy to calculate explicitly.

%%%%%%%%%%%%%%%%%%%%%%%%%%%%%%%%%%%%%%%%%%%%%%%%
\section{Introduction of the parameters \boldmath$\Atilde$, $A$, $\Btilde$, $B$}\label{S:parametres}

 \mbox{} \indent Put
 $$
 \Atilde=\max\bigl\{1, \rmh(\alpha\varepsilon)\bigr\}.
 $$
Write
$$
\varepsilon=\zeta \epsilon_1^{a_1}\cdots\epsilon_r^{a_r}
$$
with $\zeta\in K^\times_{\tors}$ and $a_i\in\Z$ for $1\le i\le r$, and put
$$
A=\max\bigl\{1,|a_1|, \dots,|a_r|\bigr\}.
$$
The upper bound
$$
\rmh(\alpha\varepsilon)\le \Newcst{majhauteuralphaepsilon} A
$$
follows from ($\ref{Equation:EstimationsTriviales}$); more precisely, we may take
$$
\cst{majhauteuralphaepsilon} =\max\bigl\{1,\; \rmh(\alpha)+\rmh(\epsilon_1)+\cdots + \rmh(\epsilon_r)\bigr\}.
$$
 The lower bound
$$
\rmh(\alpha\varepsilon)\ge \Newcst{minhauteuralphaepsilon} A
$$
follows from Lemma $\ref{Lemme:PlongementAdapte}$. Consequently,
$$
 \cst{minhauteuralphaepsilon} A\le \Atilde\le \cst{majhauteuralphaepsilon} A.
$$

Next put
$$
\beta=x- \alpha\varepsilon y
\quad\hbox{and}
\quad
\Btilde=\max\{1,\rmh(\beta)\}.
 $$
So
\begin{equation}\label{normebeta}
F_\varepsilon(x,y)=
a_0\sigma_1(\beta) \cdots \sigma_d(\beta).
\end{equation}
Since $|F_\varepsilon(x,y)|\le m$, it follows from ($\ref{Equation:lemmeA.15ST}$), ($\ref{normebeta}$) and Lemma $\ref{Lemme: LemmaA.15ST}$
that there exists $\varrho\in\ZK$ satisfying  $\rmh(\varrho)\le \Newcst{ell} \log m$ and such that
$\eta=\beta/\varrho$ belongs to the subgroup of the unit group of $\ZK$ generated by $\epsilon_1,\dots,\epsilon_r$.
Write
 $$
\eta= \epsilon_1^{b_1}\cdots\epsilon_r^{b_r}
$$
with rational numbers $b_1,\ldots,b_r$ and put
$$
B= \max\bigl\{1,\; |b_1|, |b_2|, \ldots, \; |b_r|\}.
$$
With the relation  $\beta=\varrho \eta$, we deduce from ($\ref{Equation:EstimationsTriviales}$),
$$
 \Btilde\le \Newcst{majBtilde} (B +\log m)
$$
and from Lemma $\ref{Lemme:PlongementAdapte}$ we get
$$
B\leq
\Newcst{mminBtilde}(\Btilde +\log m).
$$

By using the hypothesis $xy\not=0$, we verify that the condition $ \Q(\alpha\varepsilon)=K$ which appears in the definition of $\calE$ implies $ \Q(\beta)=\Q(\beta/\alpha\varepsilon)=K$. Consequently, for $\varphi$ and $\sigma$ in $\Phi$, we have
$$
\varphi=\sigma\;
\Longleftrightarrow \;
\varphi(\alpha\varepsilon)=\sigma(\alpha\varepsilon)\;
\Longleftrightarrow \;
\varphi(\beta)=\sigma(\beta)\;
\Longleftrightarrow \;
\sigma(\alpha\varepsilon)\varphi(\beta)=\sigma(\beta)\varphi(\alpha\varepsilon).
$$

%%%%%%%
\section{Elimination}\label{S:elimination}

\subsection{Expressions of \boldmath$x$ and $y$ in terms of $\alpha\varepsilon$ and $\beta$}\label{SS:xety}

\mbox{} \indent Let $\varphi_1,\varphi_2$ be two distincts elements of $\Phi$, {\it i.e.}, two distinct embeddings of $K$ into $\C$.
 Let us eliminate $x$ (resp. $y$) between the two equations
$$
\varphi_1(\beta)=x-\varphi_1(\alpha\varepsilon)y \quad \hbox{and}\quad
\varphi_2(\beta)=x-\varphi_2(\alpha\varepsilon)y,.
$$
This leads to
\begin{equation}\label{Equation:y}
y=\frac{\varphi_1(\beta)-\varphi_2(\beta)}{\varphi_2(\alpha\varepsilon)-\varphi_1(\alpha\varepsilon)},
\quad
x=\frac{\varphi_1(\beta)\varphi_2(\alpha\varepsilon)-\varphi_2(\beta)\varphi_1(\alpha\varepsilon)}
{\varphi_2(\alpha\varepsilon)-\varphi_1(\alpha\varepsilon)}
\cdotp
\end{equation}

%%%%%%%%%%%%%%%%%%%%%%%%%%%%%%%%%%%%%%%%%%%%%%%%
\subsection{A Siegel unit equation}\label{S:Strategie}

\mbox{} \indent Let $\varphi_1,\varphi_2,\varphi_3$ be three elements of $\Phi$, {\it i.e.}, three embeddings
of $K$ into $\C$. Write
$$
u_i=\varphi_i(\alpha\varepsilon),\quad
v_i=\varphi_i(\beta)
\qquad (i=1,2,3).
$$
We can eliminate $x$ and $y$ between the three equations
$$
\left\{
\begin{array}{lll}
\varphi_1(\beta)&=&x-\varphi_1(\alpha\varepsilon)y\\[1mm]
\varphi_2(\beta)&=&x-\varphi_2(\alpha\varepsilon)y \\[1mm]
\varphi_3(\beta)&=&x-\varphi_3(\alpha\varepsilon)y\\
\end{array}
\right.
$$
by writing that the determinant
$$
\left|
\begin{matrix}
1& \varphi_1(\alpha\varepsilon) &\varphi_1(\beta)\\
1& \varphi_2(\alpha\varepsilon) &\varphi_2(\beta)\\
1& \varphi_3(\alpha\varepsilon) &\varphi_3(\beta)\\
\end{matrix}
\right|=
\left|
\begin{matrix}
1&u_1& v_1\\
1&u_2& v_2\\
1&u_3& v_3\\
\end{matrix}
\right|
$$
is equal to $0$, and this leads to the unit equation {\it \`a la Siegel}
\begin{equation}\label{Equation:SommeSixTermes}
u_1v_2-u_1v_3+u_2v_3-u_2v_1+u_3v_1-u_3v_2=0.
\end{equation}

%%%%%%%%%%%%%%%%%%%%%%%%%%%%%%%%%%%%%%%%%%%%%%%%
\section{Four privileged embeddings}\label{S:quatrepolongements}

\mbox{} \indent
Let us denote by $\sigma_a$ (resp. $\sigma_b$) an embedding $\varphi$ of $K$ into $\C$ having the property that $|\varphi(\alpha\varepsilon)|$
 (resp. $|\varphi(\beta)|$) be maximal among all the elements $|\varphi(\alpha\varepsilon)|$ (resp. $|\varphi(\beta)|$)
when $\varphi$ runs through all the embeddings of $\Phi$.
Let us denote by $\tau_a$ (resp. $\tau_b$) an embedding $\psi$ of $K$ into $\C$ having the property that $|\psi(\alpha\varepsilon)|$
 (resp. $|\psi(\beta)|$) be minimal among all the elements $|\psi(\alpha\varepsilon)|$ (resp. $|\psi(\beta)|$)
when $\psi$ runs through all the embeddings of $\Phi$.
 We may suppose $\tau_b\not=\sigma_b $, $\tau_b\not=\overline{\sigma_b} $, together with
 $\tau_a \not=\sigma_a$, $\tau_a \not=\overline{\sigma_a}$. \medskip

 These four embeddings will be used in the unit equation in many different ways.

\section{About the parameters \boldmath $A$ and $B$}\label{S:minorationAetB}

\mbox{} \indent
It turns out that if we are under the hypothesis $\max\{A,B\}\le \kappa \log m$, we can easily prove Theorem
 $\ref{theoreme:principal}$.

\subsection{Proof of Theorem \boldmath  $\ref{theoreme:principal}$ if  $\max\{A,B\} \leq \kappa \log m$
}\label{SS:maxABsuffit}

\mbox{} \indent We indeed have a short proof of Theorem  $\ref{theoreme:principal}$ in this case.

\begin{lemme}\label{Lemme:AetBgrands}
If $\kappa$ is a real number such that
$\max\{A,B\}\le \kappa \log m$, then the conclusion of Theorem $\ref{theoreme:principal}$ holds true with a constant $\cst{kappa1}$ depending not only on $\alpha$ but also on $\kappa$.
\end{lemme}

\begin{proof}[\indent Proof. ]
In this proof, the constants
$\cst{kappa:borneAetB1}$, $ \cst{kappa:borneAetB2}$, $ \cst{kappa:borneAetB3}$, $ \cst{yinferieurakappaBbis}$, $ \cst{yinferieurakappaBter}$
depend not only on $\alpha$ but also on $\kappa$.
We use ($\ref{Equation:y}$) with the estimates
$$
 |\varphi_1(\beta)-\varphi_2(\beta)|\le e^{\Newcst{kappa:borneAetB1}B},
 \qquad
|\varphi_2(\alpha\varepsilon)-\varphi_1(\alpha\varepsilon)|\ge e^{-\Newcst{kappa:borneAetB2}A},
$$
and the estimates
$$
|\varphi_1(\beta)\varphi_2(\alpha\varepsilon)-\varphi_2(\beta)\varphi_1(\alpha\varepsilon)| \le
e^{\Newcst{kappa:borneAetB3}(A+B)},
$$
to successively obtain the upper bounds $ |y| \le m^{\Newcst{yinferieurakappaBbis} }$ and then
$|x|\le m^{ \Newcst{yinferieurakappaBter} }$.
 Therefore it follows that, if an upper bound of both $A$ and $B$ is $\kappa \log m$, the proof of Theorem $\ref{theoreme:principal}$ is secured.
\end{proof}

The goal is now to show that we can assume that we always are under the hypothesis $\max\{A,B\}\le \kappa \log m$.
 In the next two subsections, we show that this goal is achieved if we assume that an upper bound of either $A$ or $B$ is given by
a constant times $\log m$. This will mean that there is no restriction in supposing $A$ and $B$ larger than $\Newcst{kappa:minorationAetB}\log m$ with a constant $\cst{kappa:minorationAetB}$ we may assume to be sufficiently large. Once more, we could produce a convenient explicit value for this constant which will make valid the arguments which will be used.

\subsection{
Proof of Theorem \boldmath  $\ref{theoreme:principal}$  if  $A \leq \kappa'  \log m$
}\label{SS:minorationA}

\mbox{}\indent
In the next lemma, we prove that if $A \leq \kappa'  \log m$,
then we have the inequality $\max\{A,B\} \leq \kappa\log m$, whereupon Lemma  $\ref{Lemme:AetBgrands}$ states that Theorem  $\ref{theoreme:principal}$  holds true.

\begin{lemme}\label{Lemme:Agrand}
If $\kappa'>0$ is a constant such that $A\le \kappa' \log m$, then the hypothesis
$\max\{A,B\}\le \kappa \log m$ of Lemma $\ref{Lemme:AetBgrands}$ holds true with a constant $\kappa$ depending
upon $\alpha$ and $\kappa'$.
\end{lemme}

\begin{proof}[\indent Proof]
In this proof, the constants
$\cst{kappa:minorA1}$, $ \cst{kappa:minorA4}$, $ \cst{kappa:minorA3}$,
 $\cst{kappa:minorA2}$, $ \cst{kappa:bornepourAetB4}$
depend not only on $\alpha$ but also on $\kappa'$.
Let us use the unit equation ($\ref{Equation:SommeSixTermes}$) with the two embeddings $\sigma_b,\tau_b$ and a third embedding
 $\varphi$ distinct from $\sigma_b$ and $\tau_b$:
\begin{equation}\label{equation:6termespourAgrand}
\frac{\varphi(\beta) }
{\sigma_b(\beta) }
 \cdot \frac{\sigma_b(\alpha\varepsilon)-\tau_b(\alpha\varepsilon)} {\varphi(\alpha\varepsilon)- \tau_b(\alpha\varepsilon)}
-1=
-
\frac{\tau_b(\beta)}
{\sigma_b(\beta) }
 \cdot \frac{\varphi(\alpha\varepsilon)- \sigma_b(\alpha\varepsilon)} {\varphi(\alpha\varepsilon)- \tau_b(\alpha\varepsilon)}\cdotp
\end{equation}
The member on the right side is nonzero, since $\varphi\not=\sigma_b$;  it has a modulus $\leq e^{-\Newcst{kappa:minorA1}B}e^{\cst{kappa:minorA4}A}$, since
$$
|\tau_b(\beta)|\le e^{- \cst{kappa:minsigma} B},
\quad
|\sigma_b(\beta)|\ge e^{\cst{kappa:minsigma} B}
$$
and
$$
\left | \frac{\varphi(\alpha\varepsilon)- \sigma_b(\alpha\varepsilon)} {\varphi(\alpha\varepsilon)- \tau_b(\alpha\varepsilon)}
\right|\le e^{\Newcst{kappa:minorA4}A}.
$$

Put $s=r+1$, and
$$
\gamma_j= \frac{\varphi(\epsilon_j) }
{\sigma_b(\epsilon_j) },\quad c_j=b_j \quad (1\le j\le r), \quad
\gamma_s= \frac{\varphi(\varrho) }
{\sigma_b(\varrho) } \cdot \frac{\sigma_b(\alpha\varepsilon)-\tau_b(\alpha\varepsilon)} {\varphi(\alpha\varepsilon)- \tau_b(\alpha\varepsilon)},
\quad c_s=1,
$$
so
$$
\gamma_1^{c_1}\cdots\gamma_s^{c_s}=\frac{\varphi(\beta) }
{\sigma_b(\beta) }
 \cdot \frac{\sigma_b(\alpha\varepsilon)-\tau_b(\alpha\varepsilon)} {\varphi(\alpha\varepsilon)- \tau_b(\alpha\varepsilon)}\cdotp
$$
Put
$$
H_1=\cdots=H_r=\Newcst{kappa:minorA3}, \quad
H_s=\cst{kappa:minorA3} (1+ \log m),\quad
C_1=2+\frac{B}{1+\log m}\cdotp
$$
For $1\le j\le r$, we have $|c_j|\le B$ and
$$
\frac{H_j}{H_s}|c_j| \le \cst{kappa:minorA3} \frac{B}{H_s}\le \frac{B}{1+\log m}\le C_1.
$$
The upper bound
$$
\frac{H_j}{H_s}|c_j| \le C_1
$$
is still valid for $j=s$ since $c_s=1$.	 Corollary $\ref{Cor:MinorationCLL1}$ shows that a lower bound of the left member
	of ($\ref{equation:6termespourAgrand}$)
	is given by $\exp\{-\Newcst{kappa:minorA2} H_s\log C_1\}$, whereupon the above upper bound and the last lower bound lead to
$$
 \exp\{-\cst{kappa:minorA2} H_s\log C_ 1\} \le  \exp\{-\cst{kappa:minorA1}B+\cst{kappa:minorA4}A \}.
$$
Therefore
$$
\cst{kappa:minorA1}B\le \cst{kappa:minorA2} H_s\log C_1 + \cst{kappa:minorA4}A,
$$
which shows that
$B\leq \Newcst{kappa:bornepourAetB4} \log m$, which is what we wanted to establish.
\end{proof}

From now on, without loss of generality we may assume that $A$ is sufficiently large, say, $A\geq \cst{kappa:minorationAetB}\log m$.

\subsection{Proof of Theorem \boldmath  $\ref{theoreme:principal}$ if  $B \leq \kappa''  \log m$
 }\label{SS:minorationB}

\mbox{}\indent
In the next lemma, we prove that if $B \leq \kappa'' \log m$, then we have
the inequality    $\max\{A,B\} \leq \kappa \log m$, whereupon Lemma $\ref{Lemme:AetBgrands}$ states that Theorem 2 holds true.

\begin{lemme}\label{Lemme:Bgrand}
If $\kappa''>0$ is a constant such that $B\le \kappa'' \log m$, then the hypothesis
$\max\{A,B\}\le \kappa \log m$ of Lemma $\ref{Lemme:AetBgrands}$ holds true with a constant $\kappa$ depending on $\alpha$ and $\kappa''$.
\end{lemme}

\begin{proof}[\indent Proof]

In this proof, the constants
$\cst{kappa:minorB1}$, $
\cst{kappa:minorB4}$, $
 \cst{kappa:minorB3}$, $
 \cst{kappa:minorB2}$, $ \cst{kappa:bornepourAetB5}$
depend not only on $\alpha$ but also on $\kappa''$. \medskip

Assume $B\le \kappa''\log m$.
Let us use the unit equation ($\ref{Equation:SommeSixTermes}$) with the two embeddings $\sigma_a,\tau_a$
and a third one $\varphi$ distinct from $\sigma_a$ and $\tau_a$:
\begin{equation}\label{equation:6termespourBgrand}
\frac{ \varphi(\alpha\varepsilon) }
{\sigma_a(\alpha\varepsilon) }
 \cdot \frac{\sigma_a(\beta)-\tau_a(\beta)} { \varphi(\beta)- \tau_a(\beta)}-1=
-\frac{\tau_a(\alpha\varepsilon)}
{\sigma_a(\alpha\varepsilon) }
 \cdot \frac{ \varphi(\beta)- \sigma_a(\beta)} { \varphi(\beta)- \tau_a(\beta)}\cdotp
\end{equation}
The member on the right side is nonzero, since $\varphi\not=\sigma_a$;  it has a modulus $\le e^{-\Newcst{kappa:minorB1}A} e^{\cst{kappa:minorB4}B}$ since
$$
|\tau_a(\alpha\varepsilon)|\le e^{- \cst{kappa:minsigma} A},
\quad
|\sigma_a(\alpha\varepsilon)|\ge e^{\cst{kappa:minsigma} A}
\quad \mbox{and} \quad
\left | \frac{ \varphi(\beta)- \sigma_a(\beta)} { \varphi(\beta)- \tau_a(\beta)}
\right|\le e^{\Newcst{kappa:minorB4}B}.
$$
Put $s=r+1$ and
$$
\gamma_j= \frac{\varphi(\epsilon_j) }
{\sigma_a(\epsilon_j) },\quad c_j=a_j \quad (1\le j\le r), \quad
\gamma_s= \frac{\varphi(\alpha\zeta) }
{\sigma_a(\alpha\zeta) } \cdot\frac{ \sigma_a(\beta)- \tau_a(\beta)} { \varphi(\beta)- \tau_a(\beta)},
\quad c_s=1,
$$
so
$$
\gamma_1^{c_1}\cdots\gamma_s^{c_s}=
\frac{ \varphi(\alpha\varepsilon) }
{\sigma_a(\alpha\varepsilon) }
 \cdot \frac{\sigma_a(\beta)-\tau_a(\beta)} { \varphi(\beta)- \tau_a(\beta)}\cdotp
$$
Now put
$$
H_1=\cdots=H_r=\Newcst{kappa:minorB3}, \quad
H_s=\cst{kappa:minorB3}(1+ \log m),\quad
C_1= 2+\frac{A}{1+\log m}\cdotp
$$
Corollary $\ref{Cor:MinorationCLL1}$ shows that a lower bound of the left member
	of ($\ref{equation:6termespourBgrand}$)
	 is given by
 	$\exp\{-\Newcst{kappa:minorB2} H_s\log C_1\}$, whereupon the above upper bound and the last lower bound lead to
$$
\exp\{-\cst{kappa:minorB2} H_s\log C_1\}
\le \exp\{ -\cst{kappa:minorB1} A +  \cst{kappa:minorB4}B \}.
$$
 Therefore
$$
\cst{kappa:minorB1} A\le   \cst{kappa:minorB4}B + \cst{kappa:minorB2} H_s\log C_1.
$$
We conclude $A\leq \Newcst{kappa:bornepourAetB5} \log m$, which is what we wanted to achieve.
\end{proof}

\subsection{Consequences}\label{SS:consequences}

\mbox{}\indent
From now on, without loss of generality, we may assume that both $A$ and $B$ are larger than $ \cst{kappa:minorationAetB}\log m$.
Since $A$ is sufficiently large, we have $|\sigma_a(\alpha\varepsilon)|\ge 2$ and $|\tau_a(\alpha\varepsilon)|\le 1/2$.
Similarly, since $B$ is sufficiently large, we have $|\sigma_b(\beta)| \ge  2$ and $|\tau_b(\beta)| \le1/2$. \medskip

By using Lemma $\ref{Lemme:PlongementAdapte}$ with the bounds ($\ref{Equation:EstimationsTriviales}$),
we deduce that there exist some constants
$\Newcst{kappa:majsigma}$ et $\Newcst{kappa:minsigma}$ such that
\begin{equation}\label{Equation:MajMinsigma}
\left\{
\begin{array}{lll}
 e^{\cst{kappa:minsigma} A}\leq
|\sigma_a(\alpha\varepsilon)|\le e^{\cst{kappa:majsigma} A},&
\quad
e^{-\cst{kappa:majsigma} A} \le |\tau_a(\alpha\varepsilon)|\leq
e^{- \cst{kappa:minsigma} A},
\\[2mm]
e^{\cst{kappa:minsigma} B}
\leq |\sigma_b(\beta)|\le e^{\cst{kappa:majsigma} B},
&\quad
e^{-\cst{kappa:majsigma} B}
\leq |\tau_b(\beta)|\le e^{- \cst{kappa:minsigma} B}.
\end{array} \right.
\end{equation}

{\bf Remark.}
Note that the constant $\cst{kappa:minsigma}$ is sufficiently small (it comes into play as a lower bound)
while $\cst{kappa:majsigma}$ is sufficiently large (it comes into play as an upper bound).

\begin{lemme}\label{lemme:taubalphaetsigmaabeta}
We have
$$
|\tau_b(\alpha\varepsilon)| \le 2 \quad \hbox{and}\quad
|\sigma_a(\beta)| \ge e^{\Newcst{minsigmaabeta} A}.
$$
\end{lemme}

 \begin{proof}[\indent Proof]
On the one hand, we have
$$
|x-\tau_b(\alpha\varepsilon) y| =|\tau_b(\beta)| \le 1 \le |x|,
$$
and since $|x|\le |y|$, we get
$$
|\tau_b(\alpha\varepsilon) y|\le 2|x|\le 2|y|.
$$
On the other hand, upon using $|x|\le |y|$ and $|\sigma_a(\alpha\varepsilon)|\ge 2$, we find
$$
|\sigma_a(\beta) | \ge |\sigma_a(\alpha\varepsilon)y|-|x| \ge (|\sigma_a(\alpha\varepsilon)| -1) |y| \ge \frac{1}{2} |\sigma_a(\alpha\varepsilon)|
 \ge e^{\cst{minsigmaabeta} A}.
$$

\end{proof}

%%%%
\section{An upper bound for \boldmath $A$, $|x|$, $|y|$ involving $B$}\label{S:MajorationAparBtilde}

 \mbox{} \indent
\mbox{} \indent From the relation ($\ref{Equation:y}$), we will deduce in an elementary way the following upper bound.

\begin{lemme}\label{Lemme:BestGrand}
We have
$$
A\le \Newcst{AmajoreparBtilde}
B
$$
and
$$
\log\max\{|x|,\; |y| \} \le \Newcst{yinferieurakappaB}
B.
$$
\end{lemme}

\begin{proof}[\indent Proof]
Since $A$ is sufficiently large, we have $|\sigma_a(\alpha\varepsilon)|\ge 2 |\tau_a(\alpha\varepsilon)| $. We then deduce
$$
|\sigma_a(\alpha\varepsilon)-\tau_a(\alpha\varepsilon)|\ge
\frac{1}{2}|\sigma_a(\alpha\varepsilon)|.
$$
Let us use ($\ref{Equation:y}$) with $\varphi_2=\sigma_a$ and $\varphi_1=\tau_a$:
$$
y\bigl(\sigma_a(\alpha\varepsilon)-\tau_a(\alpha\varepsilon)\bigr)=
\tau_a(\beta)-\sigma_a(\beta).
$$
The very definition of $\sigma_b$ gives the upper bound
$$
|\sigma_a(\beta)-\tau_a(\beta)|\le 2|\sigma_b(\beta)|,
$$
from which we deduce
\begin{equation}\label{equation:majysigmaa}
|y\sigma_a(\alpha\varepsilon)|\le 4|\sigma_b(\beta)|.
\end{equation}
The inequalites
$$
e^{\cst{kappa:minsigma} A}\le |\sigma_a(\alpha\varepsilon) |\le | y\sigma_a(\alpha\varepsilon) | \le 4|\sigma_b(\beta)|
\le 4 e^{\cst{kappa:majsigma} B}
$$
imply $A\leq \cst{AmajoreparBtilde} B$.
From ($\ref{equation:majysigmaa}$) we deduce
$$
|y|\le 4 e^{\cst{kappa:majsigma} B}.
$$
The upper bound for $|x|$ follows from
$|x|\le |y|$.
 This concludes the proof of Lemma $\ref{Lemme:BestGrand}$.
\end{proof}

%%%%%%%%%%%%%%%%%%%%%%%%%%%%%%%%%%%%%%%%%%%%%%%%
\section{On the unicity  of \boldmath $\tau_b$ }\label{S:UniciteTaub}

We plan to show that the embedding $\tau_b$, that we have defined, is unique.

\mbox{}\indent
\begin{lemme}\label{Lemme:UniciteTau}
Suppose that there exists $\varphi\in\Phi$, $\varphi\not=\tau_b$ such that we have $|\varphi(\beta)|=|\tau_b(\beta)|$.
Then $ B \le \Newcst{UniciteTau} $.
\end{lemme}

\begin{proof}[\indent Proof]
 Suppose $|\varphi(\beta) | = | \tau_b(\beta)|$ with $\varphi\not=\tau_b$. Let us use
 ($\ref{Equation:y}$) with $\varphi_1=\varphi$, $\varphi_2=\tau_b$, under the form
 $$
-\varphi(\alpha\varepsilon) +\tau_b(\alpha\varepsilon)
=
\frac{1 }{|y|} \bigl(\varphi(\beta)-\tau_b(\beta) \bigr).
$$
We have
$$
|\varphi(\alpha\varepsilon) -\tau_b(\alpha\varepsilon) |
=
\frac{1 }{|y|} |\varphi(\beta) -\tau_b(\beta)
|\le
\frac{2|\tau_b(\beta)|}{|y|}\cdotp
$$
From
$$
|x-\tau_b(\alpha\varepsilon) y| =
|\tau_b(\beta)| \le
 e^{- \cst{kappa:minsigma} B}
  <\frac{1}{2},
$$
we deduce
$$
|\tau_b(\alpha\varepsilon) y|\ge |x|-\frac{1}{2}\ge \frac{1}{2}\cdotp
$$
Consequently,
$$
\left|
\frac{\varphi(\alpha\varepsilon) }{\tau_b(\alpha\varepsilon) }-1
\right|
\le
\frac{
2 |\tau_b(\beta) |}
{|\tau_b(\alpha\varepsilon) y|}
\le 4 |\tau_b(\beta) |
 \le 4 e^{- \cst{kappa:minsigma} B}.
 $$
Let us write the member on the left in the form $|\gamma_1^{c_1}\cdots \gamma_s^{c_s}-1|$ with
$s=r+1$,
$$
\gamma_i=\frac{\varphi(\epsilon_i)}{ \tau_b(\epsilon_i)},
\quad
c_i=a_i ,
\quad (i=1,\dots,r),
\quad
\gamma_s= \frac{\varphi(\alpha\zeta)}{\tau_b(\alpha\zeta)},
\quad
c_s=1 .
$$
From Corollary $\ref{Cor:MinorationCLL2}$ with
$$
H_1=\cdots=H_s=\Newcst{Hi},\quad C_2=A,
$$
(note that $C_2 \geq 2$ since $A$ is sufficiently large) we deduce
$$
|\gamma_1^{c_1}\cdots \gamma_s^{c_s}-1|\ge e^{-\Newcst{kappaMajB}\log A}.
$$
The above upper bound and the last lower bound lead to
$$
e^{-\cst{kappaMajB}\log A}\le
4 e^{- \cst{kappa:minsigma} B},
$$
hence
$B\le \Newcst{BmajoreparA} ( \log A)
$. Using the upper bound $A\le \cst{AmajoreparBtilde}
B$ of Lemma $\ref{Lemme:BestGrand}$ allows us to conclude.
\end{proof}

As outlined in \S$\ref{S:MajorationAparBtilde}$, the upper bound $ B \le \cst{UniciteTau} \log m$ allows to secure the proof of
Theorem $\ref{theoreme:principal}$. So Lemma
$\ref{Lemme:UniciteTau}$ allows us to suppose $|\varphi(\beta)|>|\tau_b(\beta)|$ for any $\varphi\in \Phi$ different from $\tau_b$.
In particular, the embedding $\tau_b$ is real.
This completes the proof in the case where the algebraic number field $\Q(\alpha)$ is totally imaginary. \medskip

When the algebraic number field $K$ has a unique real embedding, which is what we will assume from now on,
all embeddings different from $\tau_b$ are imaginary. In particular, since $\sigma_b\not=\tau_b$ we deduce $\sigma_b\not=\overline{\sigma_b}$.

%%%%%%%%%%%%%%%%%%%%%%%%%%%%%%%%%%%%%%%%%%%%%%%%
\section{ An upper bound for \boldmath  $B$ involving $A $ }\label{S:DeuxiemeArgumentDiophantien}

\mbox{} \indent
Let us use the strategy of \S$\ref{S:Strategie}$ with the three embeddings $\sigma_b$, $\overline{\sigma_b}$ and $\tau_b$.

\begin{lemme}\label{Lemme:BmajoreParA}
There exists an effectively computable positive constant $\Newcst{BmajoreParA}$ such that
$$
B \le \cst{BmajoreParA}A .
$$
\end{lemme}

\begin{proof}[\indent Proof]
We use the relation ($\ref{Equation:SommeSixTermes}$) with $\varphi_1=\sigma_b$, $\varphi_2=\overline{\sigma_b}$, $\varphi_3 = \tau_b$,
 in the form
$$
 \overline{\sigma_b}(\beta)
 	\bigl(\sigma_b(\alpha\varepsilon)-\tau_b(\alpha\varepsilon)\bigr)
 -\sigma_b(\beta) \bigl(\overline{\sigma_b}(\alpha\varepsilon)-\tau_b(\alpha\varepsilon)\bigr)
+\tau_b(\beta) \bigl(\overline{\sigma_b}(\alpha\varepsilon)- \sigma_b(\alpha\varepsilon)\bigr) =0,
$$
and we divide by $\sigma_b(\beta) \bigl(\overline{\sigma_b}(\alpha\varepsilon)- \tau_b(\alpha\varepsilon)\bigr)$ (which is not $0$):
\begin{equation}\label{Equation:fllMajorationBparA}
\frac{ \overline{\sigma_b}(\beta) }
{\sigma_b(\beta) }
 \cdot \frac{\sigma_b(\alpha\varepsilon)-\tau_b(\alpha\varepsilon)} { \overline{\sigma_b}(\alpha\varepsilon)- \tau_b(\alpha\varepsilon)}
-1=
-
\frac{\tau_b(\beta)}
{\sigma_b(\beta) }
 \cdot \frac{ \overline{\sigma_b}(\alpha\varepsilon)- \sigma_b(\alpha\varepsilon)} { \overline{\sigma_b}(\alpha\varepsilon)- \tau_b(\alpha\varepsilon)}\cdotp
\end{equation}
The member on the right of ($\ref{Equation:fllMajorationBparA}$) is nonzero since $\sigma_b\not= \overline{\sigma_b}$. Let us show that its modulus is bounded from above by
$$
e^{\Newcst{ExpmoinskappaA}A} e^{-\Newcst{ExpmoinskappaB}B}.
$$
On the one hand, we have
$$
|\tau_b(\beta)| \le
e^{- \cst{kappa:minsigma} B},
\quad |\sigma_b(\beta) |\ge e^{\cst{kappa:minsigma} B};
$$
on the other hand, an upper bound for the absolute logarithmic height of the number
$$
\delta=
\frac{ \overline{\sigma_b}(\alpha\varepsilon)- \sigma_b(\alpha\varepsilon)} { \overline{\sigma_b}(\alpha\varepsilon)- \tau_b(\alpha\varepsilon)}
$$
is given by $\Newcst{kappa:majhauteurdeuxiemeargt} A$, hence $|\delta|\le e^{\cst{ExpmoinskappaA} A}$. \indent

\medskip

Let us write the term
$$
\frac{ \overline{\sigma_b}(\beta) }
{\sigma_b(\beta) }
 \cdot \frac{\sigma_b(\alpha\varepsilon)-\tau_b(\alpha\varepsilon)} { \overline{\sigma_b}(\alpha\varepsilon)- \tau_b(\alpha\varepsilon)}
$$
which appears on the left side of ($\ref{Equation:fllMajorationBparA}$)
in the form $\gamma_1^{c_1}\cdots \gamma_s^{c_s}$ with $s=r+1$, and
$$
 \gamma_j=\frac{ \overline{\sigma_b}(\epsilon_j)}{\sigma_b(\epsilon_j)},
\quad
c_j=b_j \quad (j=1,\dots, r),
$$
$$
\gamma_s=\frac{\sigma_b(\alpha\varepsilon)-\tau_b(\alpha\varepsilon)} { \overline{\sigma_b}(\alpha\varepsilon)- \tau_b(\alpha\varepsilon)}\cdotp
\frac{\overline{\sigma_b}(\varrho)}{\sigma_b(\varrho)},
\quad c_s=1.
$$
We have $\rmh(\gamma_s)\le \Newcst{hauteurgamma0} A $. \medskip

With
$$
 H_1=\cdots=H_{r}=\Newcst{MinorationAparB}, \quad H_s=\cst{MinorationAparB} A,
$$
we have, using Lemma $\ref{Lemme:BestGrand}$,
$$
\max_{1\le j\le s} \frac {H_j}{ H_s} {|c_j| }\le \frac {\Newcst{csurH} B}{H_s}\cdotp
$$
Moreover, with
$$
 C_1=2+\frac{ B }{A},
$$
we obtain from Corollary $\ref{Cor:MinorationCLL2}$ that a lower bound of
the modulus of the left member of ($\ref{Equation:fllMajorationBparA}$) is given by
 $\exp\{-\Newcst{AlogB} H_s\log C_1 \}$.
 The above upper bound and the last lower bound lead to
 $$
 \exp\{-\cst{AlogB} H_s\log C_1 \}\le
\exp\{ - \cst{ExpmoinskappaB}B + \cst{ExpmoinskappaA}A\}.
 $$
  Consequently,
 $$
 \cst{ExpmoinskappaB}B \le \cst{ExpmoinskappaA}A+ \cst{AlogB} H_s\log C_1.
 $$
 Therefore,
$$
2+\frac{B}{A} =  C_1  \leq  \Newcst{MajfinaleB}\log C_1,
$$
which allows to conclude $C_1\le \Newcst{BmajoreParApluslogm}$,
 namely $
B/A \le \cst{BmajoreParA}$ .
 \end{proof}

%%%%%%%%%%%%%%%%%%%%%%%%%%%%%%%%%%%%%%%%%%%%%%%%
\section{Ultimate diophantine argument}\label{S:TroisiemeArgumentDiophantien}

\mbox{} \indent
We will use the strategy of \S$\ref{S:Strategie}$ with the three embeddings $ \sigma_a, \overline{\sigma_a},\tau_b$.

\begin{lemme}\label{Lemme:AmajoreParLogB}
We have
$$
\frac{ A+B}{\log m} \le \Newcst{kappa:AmajoreParLogB} \log
\frac{A+B}{1+\log m}\cdotp
$$
\end{lemme}

\begin{proof}[\indent Proof]
We use the relation ($\ref{Equation:SommeSixTermes}$) with
$\varphi_1=\sigma_a$, $\varphi_2=\overline{\sigma_a}$, $\varphi_3=\tau_b$ under the form
\begin{equation}\label{equation:deuxiemeargumentdiophantien}
\frac{\sigma_a(\alpha\varepsilon) \overline{\sigma_a}(\beta)}
{\overline{\sigma_a}(\alpha\varepsilon)\sigma_a(\beta)}
-1=
\frac{\tau_b(\beta)}{\sigma_a(\beta)}
\left(
\frac{\sigma_a(\alpha\varepsilon)}
{\overline{\sigma_a}(\alpha\varepsilon)} -1
\right)
+
\frac{\tau_b(\alpha\varepsilon)}{\overline{\sigma_a}(\alpha\varepsilon)}
\left(
\frac{\overline{\sigma_a}(\beta)}
{\sigma_a(\beta)} -1
\right).
\end{equation}
The left member of ($\ref{equation:deuxiemeargumentdiophantien}$) is not $0$ since $\sigma_a\not=\overline{\sigma_a}$.
 The inequalities ($\ref{Equation:MajMinsigma}$) and Lemma $\ref{lemme:taubalphaetsigmaabeta}$ indicate that we have
$$
|\tau_b(\beta)|\le e^{- \cst{kappa:minsigma} B},
 \quad
|\tau_b(\alpha\varepsilon)|\le 2,
\quad
|\sigma_a(\beta)| =
|\overline{\sigma_a}(\beta)|\ge e^{\cst{minsigmaabeta} A}.
$$
Since
$$
\max\left\{
\left|
\frac{\sigma_a(\alpha\varepsilon)}
{\overline{\sigma_a}(\alpha\varepsilon)} -1
\right|,
\;
\left|
\frac{\overline{\sigma_a}(\beta)}
{\sigma_a(\beta)} -1
\right|
\right\}
\le 2,
$$
 we deduce that each of the two terms
$$
\frac{\tau_b(\beta)}{\sigma_a(\beta)}
\left(
\frac{\sigma_a(\alpha\varepsilon)}
{\overline{\sigma_a}(\alpha\varepsilon)} -1
\right)
\quad
\hbox{and}
\quad
\frac{\tau_b(\alpha\varepsilon)}{\overline{\sigma_a}(\alpha\varepsilon)}
\left(
\frac{\overline{\sigma_a}(\beta)}
{\sigma_a(\beta)} -1
\right)
$$
of the right member of ($\ref{equation:deuxiemeargumentdiophantien}$)
has a modulus bounded  from above by $e^{- \Newcst{MajdeAparBtilde} A}$. \medskip

Let us write
$$
\frac{\sigma_a(\alpha\varepsilon) \overline{\sigma_a}(\beta)}
{\overline{\sigma_a}(\alpha\varepsilon)\sigma_a(\beta)}
=\gamma_1^{c_1}\cdots \gamma_s^{c_s}
$$
with $ s=r+1$ and
$$
 \gamma_i=\frac{ \sigma_a(\epsilon_i)}{ \overline{\sigma_a}(\epsilon_i)}, \quad 
 c_i=a_i-b_i
 \quad (i=1,\dots, r),
 \quad
 \gamma_s=\frac{\sigma_a(\alpha\zeta) \overline{\sigma_a}(\varrho) }{\overline{\sigma_a}(\alpha\zeta) \sigma_a(\varrho) } ,
 \quad c_s=1.
 $$
 The absolute logarithmic height of $\gamma_s$ is bounded  from above by
 $$
 \rmh(\gamma_s)\le \Newcst{kappa:htgammas} \log m.
 $$
Let us also write
$$
 H_1=\cdots=H_{r}=\Newcst{MinorationBparA}, \quad H_s=\cst{MinorationBparA} (1+ \log m), \quad
 C_1=2+\frac{A+B}{1+\log m}\cdotp
$$
Corollary $\ref{Cor:MinorationCLL2}$ implies that the modulus of the left member of
 ($\ref{equation:deuxiemeargumentdiophantien}$) has a lower bound given by
 $\exp\{-\Newcst{AlogB1} H_s\log C_1 \}$.
   The above upper bound and the last lower bound lead to
   $$
   \exp\{-\cst{AlogB1} H_s\log C_1 \}
   \le
   2
\exp\left\{-
\cst{MajdeAparBtilde} A
\right\}.
 $$
  Consequently,
 $$
\frac{A+B}{\log m} \le \Newcst{AlogB2} \log C_1.
 $$
This completes the proof of Lemma $\ref{Lemme:AmajoreParLogB}$.
 \end{proof}

%%%%%%%%%%%%%%%%%%%%%%%%%%%%%%%%%%%%%%%%%%%%%%%%
\section{End of the proof of Theorem \boldmath $\ref{theoreme:principal}$ }\label{S:FinDemonstrationTheoreme}

\mbox{} \indent
Thanks to Lemma $\ref{Lemme:BestGrand}$ we have
$$
A\le \cst{AmajoreparBtilde}
B
$$
while from Lemma $\ref{Lemme:BmajoreParA}$ we deduce
$$
B \le \cst{BmajoreParA} A ,
$$
and from Lemma $\ref{Lemme:AmajoreParLogB}$ we have
$$
\frac{ A+B}{\log m} \le \cst{kappa:AmajoreParLogB} \log
\frac{A+B}{ 1+ \log m}\cdotp
$$
We conclude $\max\{A,B\}\le \Newcst{kappafinal} \log m$, and this is what we wanted to show in order to apply Lemma $\ref{Lemme:AetBgrands}$.

\end{document}